\documentclass[french]{smfart}
\usepackage{url}
\usepackage[frenchb]{babel}
\usepackage{amsopn} 
\usepackage{amsmath,upref}
\usepackage{mathrsfs}
\usepackage{amssymb}
\usepackage[all,2cell,emtex]{xy}
\input xypic 

\newtheorem{thm}{Th\'eor\`eme}[section]
\newtheorem{prop}[thm]{Proposition}
\newtheorem{lemme}[thm]{Lemme}
\newtheorem{cor}[thm]{Corollaire}

\theoremstyle{remark} 
\newtheorem{rem}[thm]{Remarque}

\newtheorem{paragr}[thm]{}

\theoremstyle{definition} 

\theoremstyle{plain}

\newcommand{\derR}{\mathbf{R}}

\newcommand{\To}{\longrightarrow}

\newcommand{\Hom}{\operatorname{\mathrm{Hom}}}
\newcommand{\sHom}{\mathit{Hom}}

\newcommand{\op}[1]{{#1}^{\mathit{op}}}

\newcommand{\ho}{\operatorname{\mathrm{Ho}}}

\newcommand{\A}{\mathcal{A}}

\newcommand{\derL}{\mathbf{L}}

\newcommand{\esp}{\mathcal{E}}
\newcommand{\Sp}{\mathit{Sp}}
\newcommand{\Gm}{\mathbf{G}_m}

\renewcommand{\AA}{\mathbf{A}}
\newcommand{\PP}{\mathbf{P}}
\newcommand{\KV}{{I\mspace{-6.mu}K}}
\newcommand{\KH}{\mathit{KH}}
\newcommand{\TT}{\mathit{B}}
\newcommand{\SH}{\mathcal{S}\mspace{-1.45mu}\mathcal{H}}
\newcommand{\ZZ}{\mathbf{Z}}

\newcommand{\KGL}{\mathit{KGL}}

\newcommand{\BGL}{\mathit{BGL}}

\def\TO#1{\mathrel{\hbox to #1mm{\rightarrowfill}}}
\def\OT#1{\mathrel{\hbox to #1mm{\leftarrowfill}}}

\renewcommand{\theequation}{\arabic{equation}}
\numberwithin{equation}{thm}

\title[Descente par \'eclatements]{Descente par \'eclatements\\
en $K$-th\'eorie invariante par homotopie}
\author[D.-C. Cisinski]{Denis-Charles Cisinski}
\address{Universit\'e Paul Sabatier\\
Institut de Math\'ematiques de Toulouse\\
118 route de Narbonne\\
31062 Toulouse cedex 9\\France}
\email{denis-charles.cisinski@math.univ-toulouse.fr}
\urladdr{http://www.math.univ-toulouse.fr/~dcisinsk/}

\begin{document}
\begin{abstract}
Ces notes donnent une preuve de la repr\'esentabilit\'e de la $K$-th\'eorie
invariante par homotopie dans la cat\'egorie homotopique stable des sch\'emas
(r\'esultat annonc\'e par Voevodsky). On en d\'eduit, gr\^ace au th\'eor\`eme
de changement de base propre en th\'eorie de l'homotopie stable des
sch\'emas, un th\'eor\`eme de descente par \'eclatements en
$K$-th\'eorie invariante par homotopie.
\end{abstract}
\begin{altabstract}
These notes give a proof of the representability of homotopy invariant
$K$-theory in the stable homotopy category of schemes (which was announced by Voevodsky).
One deduces from the proper base change theorem in stable homotopy theory
of schemes a descent by blow-ups theorem for homotopy invariant
$K$-theory.
\end{altabstract}
\maketitle
\section*{Introduction}
La descente par \'eclatements pour la $K$-th\'eorie
invariante par homotopie a \'et\'e prouv\'ee par C.~Haesemeyer~\cite{Has}
pour les sch\'emas de caract\'eristique nulle. Il s'agit de l'un des ingr\'edients
de la preuve de la conjecture de Weibel~\cite[Question 2.9]{Wei0}, d'annulation de la $K$-th\'eorie
n\'egative au del\`a de la dimension de Krull,
pour les sch\'emas de caract\'eristique nulle; voir \cite{CHSW}.

Nous prouvons ici la descente par \'eclatements pour la $K$-th\'eorie invariante par homotopie
pour les sch\'emas noeth\'eriens de dimension de Krull finie (sans aucune restriction
sur la caract\'eristique). L'argument invoqu\'e ici n'utilise pas de r\'esolution des singularit\'es, mais
consiste \`a d\'emontrer un r\'esultat int\'eressant en soit, annonc\'e
par V.~Voevodsky~\cite[Th\'eor\`eme 6.9]{voe}: la repr\'esentabilit\'e
de la $K$-th\'eorie invariante par homotopie dans la cat\'egorie homotopique stable des sch\'emas
de Morel et Voevodsky par le $S^1\wedge\Gm$-spectre $\KGL$.
La descente par \'eclatements en $K$-th\'eorie
invariante par homotopie est alors un simple corollaire
des th\'eor\`emes de changement de base lisse et
de changement de base propre dans le cadre motivique,
d\'emontr\'es par J.~Ayoub~\cite{ay1}, et l\'eg\`erement g\'en\'eralis\'es
dans \cite{CD3}.

Pour terminer ces notes,
en utilisant le fait que la $K$-th\'eorie non-connective est invariante
par homotopie modulo la $p$-torsion pour les sch\'emas de caract\'eristique
\smash{$p>0$}, on en d\'eduit,
sous l'hypoth\`ese de l'existence d'une r\'esolution des singularit\'es \emph{locale}
pour les $k$-sch\'emas de type fini, la conjecture de Weibel en caract\'eristique positive modulo
$p$-torsion (d'apr\`es T.~Geisser \& L.~Hesselholt~\cite{GH}, ainsi que
A.~Krishna~\cite{Kri}, on peut aussi avoir des r\'esultats \`a coefficients entiers, mais
sous l'hypoth\`ese de l'existence d'une r\'esolution des singularit\'es \emph{globale}).
Signalons que, depuis la r\'edaction de la premi\`ere version de ces notes,
la conjecture de Weibel en caract\'eristique $p>0$ modulo $p$-torsion
a \'et\'e d\'emontr\'ee inconditionnellement
par Shane~Kelly \cite{shane}, en utilisant les r\'esultats de
repr\'esentabilit\'e et de descente prouv\'es ici, ainsi que les raffinements apport\'es par O.~Gabber
\`a la th\'eorie des alt\'erations de de~Jong.

Dans ce qui suit, tous les sch\'emas seront noeth\'eriens,
de dimension de Krull finie.

\section{Spectres invariants par homotopie}

\begin{paragr}
Si $E$ est un $S^1$-spectre, on notera $\pi_n(E)$ son $n$-\`eme
groupe d'homotopie stable. Lorsque nous aurons envie
d'insister sur le point de vue cohomologique, nous \'ecrirons
$$H^n(E)=\pi_{-n}(E)\, .$$
\end{paragr}

\begin{paragr}
Soit $S$ un sch\'ema. On consid\`ere la cat\'egorie
$\esp(S)$
des pr\'efaisceaux simpliciaux sur la cat\'egorie des $S$-sch\'ema lisses,
munie de la structure de cat\'egorie de mod\`eles projective
(voir par exemple \cite[Proposition 4.4.16]{ay2}).
On d\'esigne par $\esp_\bullet(S)$ sa variante point\'ee.
On note $\Sp_{S^1}(S)$ la cat\'egorie de mod\`eles stable
des $S^1$-spectres (sym\'etriques) dans $\esp_\bullet(S)$ (ou encore, de mani\`ere
\'equivalente, la cat\'egorie de mod\`eles projective des pr\'efaisceaux en $S^1$-spectres
(sym\'etriques) sur la cat\'egorie des $S$-sch\'emas lisses).
La cat\'egorie homotopique correspondante $\ho(\Sp_{S^1}(S))$
est canoniquement
munie d'une structure de cat\'egorie triangul\'ee engendr\'ee
par ses objets compacts. Une famille g\'en\'e\-ra\-trice est donn\'ee
par la collection des (suspensions des) objets de la forme $\Sigma^\infty(X_+)$,
o\`u $X$ parcourt la classe des $S$-sch\'emas lisses.
Si $E$ est un $S^1$-spectre dans $\esp_\bullet(S)$, on note, pour tout
entier $n$ et pour tout $S$-sch\'ema lisse $X$,
$$H^n(E(X))\simeq \Hom_{\ho(\Sp_{S^1}(S))}(\Sigma^\infty(X_+),\Sigma^nE)$$
le $n$-\`eme groupe de cohomologie de $X$ \`a coefficients dans $E$.
Un morphisme $E\To F$ de $\ho(\Sp_{S^1}(S))$ est un isomorphisme
si et seulement si, pour tout $S$-sch\'ema lisse $X$, et pour tout entier $n$,
il induit un isomorphisme de groupes ab\'eliens $H^n(E(X))\simeq H^n(F(X))$.

On dispose aussi de la sph\`ere de Tate $S^1\wedge\Gm$
(o\`u le groupe multiplicatif $\Gm$ est consid\'er\'e ici comme un pr\'efaisceau
point\'e par $1$). On notera $T$ tout remplacement cofibrant
de $S^1\wedge\Gm$ dans $\esp_\bullet(S)$ de sorte que $T\simeq S^1\wedge\Gm$
dans $\ho(\esp_\bullet(S))$.
\end{paragr}

\begin{paragr}\label{defA1loc}
Un pr\'efaisceau en $S^1$-spectres $E$ est dit \emph{invariant par homotopie}
si, pour tout $S$-sch\'ema lisse $X$, et pour tout entier $n$, la projection
$X\times\AA^1\To X$ induit un isomorphisme
$$H^n(E(X))\simeq H^n(E(X\times\AA^1))\, .$$
On note $\ho_{\AA^1}(\Sp_{S^1}(S))$ la sous-cat\'egorie pleine des
$S^1$-spectres invariants par homotopie. Cette derni\`ere cat\'egorie
peut \^etre d\'ecrite comme une localisation de la cat\'egorie
triangul\'ee $\ho(\Sp_{S^1}(S))$ comme suit. D\'esignons par $\mathcal{A}$
la sous-cat\'egorie localisante (i.e. stables par petites sommes quelconques,
par suspensions et cosuspensions, ainsi que par extensions) de $\ho(\Sp_{S^1}(S))$
engendr\'ee par les c\^ones des morphismes
$\Sigma^\infty(X\times\AA^1_+)\To\Sigma^\infty(X_+)$
(induits par les projections $X\times\AA^1\To X$).
Alors le foncteur vers le quotient de Verdier correspondant
$$\ho(\Sp_{S^1}(S))\To\ho(\Sp_{S^1}(S))/\mathcal{A}$$
admet un adjoint \`a droite pleinement fid\`ele dont l'image essentielle
est pr\'eci\-s\'e\-ment la sous-cat\'egorie pleine de $\ho(\Sp_{S^1}(S))$
form\'ee des $S^1$-spectres invariants par homotopie.
On peut calculer l'adjoint \`a gauche du foncteur d'inclusion
$$i:\ho_{\AA^1}(\Sp_{S^1}(S))\To\ho(\Sp_{S^1}(S))$$
de la mani\`ere suivante.
On rappelle que l'on dispose d'un $S$-sch\'ema cosimplicial
$\Delta^\bullet_S$ d\'efini par
$$\Delta^n_S=S\times\mathrm{Spec}\, \big(\mathbf{Z}[t_0,\ldots,t_n]/(t_0+\dots+t_n-1)\big)\, ,$$
et qu'on a des isomorphismes (non canoniques) $\Delta^n_S\simeq\AA^n_S$.
Si $E$ est un pr\'e\-fais\-ceau de $S^1$-spectres sur la cat\'egorie des $S$-sch\'emas lisses,
on note $R_{\AA^1}(E)$ le $S^1$-spectre d\'efini par la formule
$$R_{\AA^1}(E)=\underset{\quad\Delta^n\in\op{\Delta}}{\derL\varinjlim}\, \derR\sHom(\Sigma^\infty({\Delta^n_{S\, +}}),E)$$
o\`u $\derR\sHom$ d\'esigne le hom interne de la cat\'egorie $\ho(\Sp_{S^1}(S))$, et o\`u
la colimite homotopique est index\'ee par la cat\'egorie oppos\'ee de la cat\'egorie des simplexes).
On v\'erifie imm\'edia\-te\-ment que $R_{\AA^1}(E)$ est invariant par homotopie, et que
le foncteur $R_{\AA^1}$ est l'adjoint \`a gauche du foncteur d'inclusion ci-dessus.
Autrement dit, le morphisme canonique
$$E\To R_{\AA^1}(E)$$
est le morphisme universel de $E$ vers un pr\'efaisceau de $S^1$-spectres
invariant par homotopie.

On dira qu'un morphisme de $\ho(\Sp_{S^1}(S))$ est une \emph{$\AA^1$-\'equivalence}
si son image par le foncteur $R_{\AA^1}$ est un isomorphisme.
On peut donc d\'ecrire la cat\'egorie $\ho_{\AA^1}(\Sp_{S^1}(S))$
comme la localisation de la cat\'egorie $\ho(\Sp_{S^1}(S))$ par la classe
des $\AA^1$-\'equivalences (le foncteur $R_{\AA^1}$ \'etant alors le foncteur
de localisation canonique).
\end{paragr}

\begin{lemme}\label{A1eqHomint}
Soit $C$ un objet compact de $\ho(\Sp_{S^1}(S))$. Le foncteur
$$\derR\sHom(C,-):\ho(\Sp_{S^1}(S))\To\ho(\Sp_{S^1}(S))$$
respecte les $\AA^1$-\'equivalences. En particulier, pour tout
pr\'efaisceau de $S^1$-spec\-tres $E$, on a un isomorphisme
canonique
$$R_{\AA^1}(\derR\sHom(C,E))\simeq\derR\sHom(C,R_{\AA^1}(E))\, .$$
\end{lemme}

\begin{proof}
Pour un $S$-sch\'ema lisse $X$ et un pr\'efaisceau de $S^1$-spectres $E$, on notera
$$E^X=\derR\sHom(\Sigma^\infty(X_+),E)\, .$$
Pour tout pr\'efaisceau de $S^1$-spectres $E$, la projection $\AA^1_S\To S$
induit une $\AA^1$-\'equivalence
$$E\To E^{\AA^1_S}\, .$$
En effet, le morphisme de multiplication $\mu:\AA^1_S\times_S\AA^1_S\To \AA^1_S$
induit un morphisme
$$E^{\AA^1_S}\To E^{\AA^1_S\times_S\AA^1_S}=(E^{\AA^1_S})^{\AA^1_S}\, ,$$
d'o\`u un morphisme
$$h:\Sigma^\infty(\AA^1_{S\, +})\wedge^\derL E^{\AA^1_S}\To E^{\AA^1_S}\, ,$$
lequel est une $\AA^1$-homotopie de l'identit\'e de $E^{\AA^1}$
avec le morphisme compos\'e
$$E^{\AA^1_S}\To E \To E^{\AA^1_S}\, .$$

D\'esignons par $\A'$  la sous-cat\'egorie localisante de $\ho(\Sp_{S^1}(S))$
engendr\'ee par les c\^ones de morphismes de la forme
$C\To C^{\AA^1}$, pour $C$ un objet compact. La remarque pr\'ec\'edente
montre que $\A'\subset \A$ (o\`u $\A$ est la sous-cat\'egorie localisante
introduite au num\'ero \ref{defA1loc}).
On a en fait l'\'egalit\'e $\A=\A'$. En effet, le morphisme de multiplication
$$\Sigma^\infty(\AA^1_{S\, +})\wedge\Sigma^\infty(\AA^1_{S\, +})\To\Sigma^\infty(\AA^1_{S\, +})$$
induit, pour tout $S$-sch\'ema lisse $X$, un morphisme
$$\Sigma^\infty(\AA^1_{X\, +})\To\Sigma^\infty(\AA^1_{X\, +})^{\AA^1_S}\, ,$$
lequel est une $\AA^1$-homotopie (en termes d'espace des chemins)
de l'identit\'e de $\Sigma^\infty(\AA^1_{X\, +})$
avec le morphisme compos\'e
$$\Sigma^\infty(\AA^1_{X\, +})\To \Sigma^\infty(X_+)\To \Sigma^\infty(\AA^1_{X\, +})\, .$$
Le m\^eme type de consid\'erations montre plus g\'en\'eralement que, pour tout entier $n\geq 0$, et pour
tout objet compact $C$ de $\ho(\Sp_{S^1}(S))$, le c\^one du morphisme canonique
$C\To C^{\AA^n_S}$ est dans $\A=\A'$.
\'Etant donn\'e que $\Sigma^\infty(\AA^1_{S\, +})$ est compact, et que
les objets compacts de $\ho(\Sp_{S^1}(S))$ sont stables par produit tensoriel
(smash produit), le foncteur $E\longmapsto E^{\AA^1_S}$
commute aux sommes quelconques, d'o\`u l'on d\'eduit que la classe $\A'$
contient en fait tous les c\^ones de morphismes
de la forme $E\To E^{\AA^1_S}$ pour tout $E$.

Soit $C$ un objet compact de $\ho(\Sp_{S^1}(S))$. Pour montrer que
le foncteur correspondant $\derR\sHom(C,-)$ respecte les $\AA^1$-\'equivalences, il suffit donc \`a pr\'esent
de v\'erifier que ce foncteur envoie $\A'$ dans $\A'$.
Soit $E$ un objet de $\ho(\Sp_{S^1}(S))$. On v\'erifie aussit\^ot que
$$\derR\sHom(C,E^{\AA^1_S})\simeq\derR\sHom(C,E)^{\AA^1_S}\, ,$$
d'o\`u on d\'eduit
la premi\`ere assertion du lemme.

Consid\'erons \`a pr\'esent deux objets $C$ et $E$ de $\ho(\Sp_{S^1}(S))$, avec $C$ compact.
Les $\AA^1$-\'equivalences \'etant stables par produit tensoriel, le spectre
$$\derR\sHom(C,R_{\AA^1}(E))$$
est invariant par homotopie, ce qui implique que
le morphisme canonique
$$\derR\sHom(C,E)\To\derR\sHom(C,R_{\AA^1}(E))$$
induit un morphisme non moins canonique
$$R_{\AA^1}(\derR\sHom(C,E))\To\derR\sHom(C,R_{\AA^1}(E))\, .$$
Le fait que ce dernier soit un isomorphisme r\'esulte du fait que le morphisme canonique
$\derR\sHom(C,E)\To\derR\sHom(C,R_{\AA^1}(E))$ est une $\AA^1$-\'equivalence,
ce qui se voit de la mani\`ere suivante: l'objet $C$ \'etant compact
dans une cat\'egorie homotopique stable,
le foncteur $\derR\sHom(C,-)$ commute aux colimites homotopiques,
et donc ce morphisme est une colimite homotopique
de morphismes de la forme
$$\derR\sHom(C,E)\To\derR\sHom(C,E^{\AA^n_S})=\derR\sHom(C,E)^{\AA^n_S}\, ,$$
lesquels sont tous des $\AA^1$-\'equivalences.
\end{proof}

\section{Repr\'esentabilit\'e de la $K$-th\'eorie invariante par homotopie}

\begin{paragr}\label{Bott}
On note $K$ l'objet de $\ho(\Sp_{S^1}(S))$ repr\'esentant la $K$-th\'eorie
au sens de Thomason et Trobaugh \cite[D\'efinition 3.1]{TT}
(laquelle, d'apr\`es \cite[Proposition 3.10]{TT},
co\"\i ncide avec celle de Quillen pour les sch\'emas admettant une
famille ample de fibr\'es en droites, en particulier, pour les sch\'emas
affines).
On a donc, pour tout $S$-sch\'ema lisse $X$,
un isomorphisme canonique de groupes ab\'eliens
$$\Hom_{\ho(\Sp_{S^1}(S))}(\Sigma^n \Sigma^\infty(X_+),K(X))\simeq K_n(X)$$
(avec $K_n(X)=0$ si $n<0$).
Le spectre de $K$-th\'eorie est un spectre en anneaux (un mono\"\i de commutatif
dans $\ho(\Sp_{S^1}(S))$) via le produit tensoriel (d\'eriv\'e) des complexes parfaits.

On d\'efinit la \emph{$K$-th\'eorie invariante par homotopie na\"\i ve} $\KV$ par
la formule
$$\KV=R_{\AA^1}(K)\, .$$
La structure de spectre en anneaux sur $K$ induit canoniquement une
telle structure sur $\KV$, de telle fa\c con que le morphisme canonique
$K\To \KV$ soit un morphisme de spectres en anneaux.

On choisit une pr\'esentation du groupe multiplicatif
$$\Gm=S\times \mathrm{Spec}\, \mathbf{Z}[t,t^{-1}]\, ,$$
et on note $b\in K_1(\Gm)$ la classe associ\'ee \`a
la section inversible $t$, laquelle
correspond \`a un morphisme
$$b:T=S^1\wedge\Gm\To \derR\Omega^\infty(K)$$
dans $\ho(\esp_\bullet(S))$.

On dispose alors du cup produit par la classe $b$ en $K$-th\'eorie et en
$K$-th\'eorie invariante par homotopie na\"\i ve.
$$b\cup:T\wedge^\derL K\xrightarrow{b\wedge^\derL 1_K} K\wedge^\derL K\xrightarrow{ \ \mu \ } K $$
$$b\cup:T\wedge^\derL \KV\xrightarrow{b\wedge^\derL 1_\KV} \KV\wedge^\derL \KV\xrightarrow{ \ \mu \ } \KV $$
\end{paragr}

\begin{paragr}\label{paragr:discuterelevementdew}
Comme on le voit, on aura donc \`a consid\'erer des couples $(E,w)$, o\`u
$E$ est un pr\'efaisceau de $S^1$-spectres, et o\`u $w:T\wedge^\derL E\To E$
est un morphisme de $\ho(\Sp_{S^1}(S))$. On peut toujours supposer que $E$
est \`a la fois fibrant et cofibrant. Cela implique alors que les
propri\'et\'es suivantes sont v\'erifi\'ees (on rappelle qu'on a rendu $T$ cofibrant
dans la cat\'egorie $\esp_\bullet(S)$ des pr\'efaisceaux simpliciaux point\'es):
\begin{itemize}
\item[(a)] le morphisme canonique $T\wedge^\derL E\To T\wedge E$ est
un isomorphisme dans $\ho(\Sp_{S^1}(S))$ (puisque $E$ est cofibrant);
\item[(b)] le morphisme $w$ se rel\`eve en un morphisme $\underline{w}:T\wedge E\To E$
(car $T\wedge E$ est cofibrant et $E$ est fibrant).
\end{itemize}
Dans la pratique, on notera par abus $\underline{w}=w$, ce qui est justif\'e par
la proposition \ref{prop:independancedurelevement} et la remarque \ref{rem:independancedurelevement}
ci-dessous.

Si $E$ et $E'$ sont deux pr\'efaisceaux de $S^1$-spectres chacun munis de morphismes
$w:T\wedge E\To E$ et $w':T\wedge E'\To E'$, un morphisme de pr\'efaisceaux de $S^1$-spectres
$f:E\To E'$ sera dit \emph{$T$-\'equivariant} si le diagramme suivant commute.
$$\xymatrix{
T\wedge E\ar[r]^w\ar[d]_{T\wedge f}&E\ar[d]^{f}\\
T\wedge E'\ar[r]^{w'}&E'}$$
\end{paragr}

\begin{prop}\label{prop:independancedurelevement}
Soit $E$ un pr\'efaisceau de $S^1$-spectres \`a la fois fibrant et cofib\-rant,
et soient $w_i:T\wedge E\To E$, $i=0,1$,
deux morphismes de pr\'efaisceaux de $S^1$-spectres
tels que $w_0=w_1$ dans $\ho(\Sp_{S^1}(S))$. Alors il existe un
pr\'efaisceau de $S^1$-spectres cofibrant $E'$, un morphisme
de pr\'efaisceaux $w':T\wedge E'\To E'$, et deux \'equivalences faibles
$T$-\'equivariantes $f_i:E\To E'$, pour $E$ muni de $w_i$, $i=0,1$.
\end{prop}

\begin{proof}
Notons $\Delta^1$ l'intervalle simplicial (vu comme un pr\'e\-fais\-ceau constant),
et posons $I=\Sigma^\infty(\Delta^1_+)$. La diagonale de $\Delta^1$ et
le morphisme $\Delta^1\To\Delta^0$ induisent une structure d'alg\`ebre
de Hopf sur $I$; on notera ici $\Delta:I\To I\wedge I$ la comultiplication,
et $\eta:I\To \Sigma^\infty(S_+)$ la co-unit\'e. Les inclusions $\{i\}\subset\Delta^1$,
$i=0,1$, induisent des cofibrations triviales $d_i:\Sigma^\infty(S_+)\To I$, $i=0,1$, de sorte que
$\eta d_i=1_I$. En particulier, comme $E$ est cofibrant, $T\wedge E\wedge I$ est un
objet cylindre de $T\wedge E$ au sens des cat\'egories de mod\`eles. Comme $E$ est
aussi fibrant, on en d\'eduit qu'il existe un morphisme
$$h:T\wedge E\wedge I\To E$$
tel que $h(1_E\wedge d_i)=w_i$ pour $i=0,1$. On pose $E'=E\wedge I$, et on d\'efinit
$w'$ comme le compos\'e
$$T\wedge E\wedge I\xrightarrow{T\wedge 1_E\wedge \Delta}T\wedge E\wedge I\wedge I
\xrightarrow{h\wedge 1_I}E\wedge I\, .$$
On d\'efinit enfin $f_i=1_E\wedge d_i:E\To E\wedge I$ pour $i=0,1$, et on v\'erifie
aussit\^ot que cela donne la construction voulue.
\end{proof}

\begin{rem}\label{rem:independancedurelevement}
Pour tout morphisme de pr\'efaisceaux de $S^1$-spectres $w:T\wedge E\To E$
et pour toute cofibration triviale de but fibrant $i:E\To E'$, comme
$T\wedge i:T\wedge E\To T\wedge E'$ est toujours une cofibration triviale,
on peut toujours trouver $w':T\wedge E'\To E'$ de sorte que $i$ soit un
morphisme $T$-\'equivariant.
En particulier, dans la conclusion de la proposition \ref{prop:independancedurelevement},
on peut toujours imposer que $E'$ soit aussi fibrant.

Une variation de l'argument pr\'ec\'edent s'applique \`a la cat\'egorie
de mod\`eles obtenue comme la localisation de Bousfield \`a gauche
$L_{\AA^1}\Sp_{S^1}(S)$ de $\Sp_{S^1}(S)$ par les morphismes de la forme
$$\Sigma^n\Sigma^\infty(\AA^1\times X_+)\To\Sigma^n\Sigma^\infty(X_+)$$
pour $X$ lisse sur $S$ et $n$ un entier (les \'equivalences faibles
correspondantes sont alors les morphismes induisant un isomorphisme
dans $\ho_{\AA^1}(\Sp_{S^1}(S))$). On en d\'eduit que, pour tout
pr\'efaisceau de $S^1$-spectres $E$ muni d'un morphisme $w:T\wedge E\To E$,
le morphisme canonique $E\To R_{\AA^1}(E)$ peut \^etre r\'ealis\'e comme
une cofibration triviale de but fibrant dans $L_{\AA^1}\Sp_{S^1}(S)$,
de sorte qu'il existe un morphisme $w_{\AA^1}:T\wedge R_{\AA^1}(E)\To R_{\AA^1}(E)$
faisant de $E\To R_{\AA^1}(E)$ un morphisme $T$-\'equivariant.

De m\^eme, pour toute fibration triviale $p:E'\To E$ de source cofibrante,
comme $T\wedge E'$ est encore cofibrant, on voit
qu'il existe un morphisme $w':T\wedge E'\To E'$ tel que le morphisme $p$ soit
$T$-\'equivariant.
\end{rem}

\begin{paragr}\label{consTT}
Consid\'erons \`a pr\'esent un objet $E$ dans $\Sp_{S^1}(S)$, muni
d'un morphisme $w:T\wedge E\To E$. On notera par abus encore $w:T\wedge^\derL E\To E$
le morphisme induit par le morphisme canonique $T\wedge^\derL E\To T\wedge E$.
\`A une telle donn\'ee, nous allons associer deux nouveaux spectres, not\'es
respectivement $E^\TT$ et $E^\sharp$, et nous verrons ensuite
qu'ils sont $\AA^1$-\'equivalents; voir le corollaire \ref{A1invTT3}.
Ces constructions ont lieu dans la cat\'egorie de mod\`eles des
pr\'efaisceaux de $S^1$-spectres (en particulier, on utilisera
un certain nombre de fois des remplacements fibrants
et des remplacements cofibrants). On n'exprimera cependant ces constructions
que dans le language des foncteurs d\'eriv\'es laissant au lecteur le soin
de v\'erifier qu'une utilisation r\'ep\'et\'ee de la remarque \ref{rem:independancedurelevement}
donne bien un sens aux objets consid\'er\'es.

On commence par la construction de Bass-Thomason-Trobaugh $E^\TT$.
Nous allons cons\-truire par r\'ecurrence une famille de morphismes
$$E=F_0\To F_{-1}\To \cdots\To F_k\To F_{k-1}\To\cdots\, .$$
Pour un objet $C$ de $\ho(\Sp_{S^1}(S))$, on note $V(C)$
l'objet d\'efini par le carr\'e homotopiquement cocart\'esien suivant.
$$\xymatrix{
C\ar[r]\ar[d]&\derR\sHom(\Sigma^\infty(\AA^1_{S\, +}),C)\ar[d]\\
\derR\sHom(\Sigma^\infty(\AA^1_{S\, +}),C)\ar[r]& V(C)
}$$
Par fonctorialit\'e, le carr\'e commutatif
$$\xymatrix{
\Gm\ar[r]\ar[d]&\AA^1_S\ar[d]\\
\AA^1_S\ar[r]& S
}$$
(o\`u l'une des inclusions $\Gm\subset\AA^1$ correspond \`a la pr\'esentation
$\AA^1=\mathrm{Spec}(\ZZ[t])$, et l'autre \`a la pr\'esentation
$\AA^1=\mathrm{Spec}(\ZZ[t^{-1}])$)
induit un carr\'e commutatif
$$\xymatrix{
C\ar[r]\ar[d]&\derR\sHom(\Sigma^\infty(\AA^1_{S\, +}),C)\ar[d]\\
\derR\sHom(\Sigma^\infty(\AA^1_{S\, +}),C)\ar[r]& \derR\sHom(\Sigma^\infty(\mathbf{G}_{m +}),C)
}$$
et par l\`a, un morphisme canonique
$$V(C)\To\derR\sHom(\Sigma^\infty(\mathbf{G}_{m +}),C)\, .$$
On note $U(C)$ la fibre homotopique de ce dernier, ce qui donne, par d\'efinition, un
triangle distingu\'e canonique
$$U(C)\To V(C)\To\derR\sHom(\Sigma^\infty(\mathbf{G}_{m +}),C)\To \Sigma U(C) \, .$$
Si en outre on dispose d'un morphisme
$w:T\wedge^\derL C\To C$, alors, pour tout objet $A$ de $\ho(\Sp_{S^1}(S))$, on
a un morphisme canonique
$$T\wedge^\derL\derR\sHom(A,C)\To\derR\sHom(A,C)$$
correspondant par adjonction \`a l'image par le foncteur $\derR\sHom(A,-)$
du morphisme $C\To\derR\sHom(T,C)$ induit par $w$.
$$\derR\sHom(A,C)\To\derR\sHom(A,\derR\sHom(T,C))\simeq\derR\sHom(T\wedge^\derL A,C)$$
Cette construction \'etant fonctorielle en $A$ (et le smash produit par $T$ commutant aux
colimites homotopiques), on en d\'eduit alors des morphismes naturels
$$T\wedge^\derL V(C)\To V(C) \quad \text{et}\quad T\wedge^\derL U(C)\To U(C)\, .$$
Le morphisme $w$ permet en outre de produire un morphisme canonique
$$C\To U(C)$$
obtenu comme le morphisme compos\'e
$$C \To\derR\sHom(T,C)
\subset\Sigma^{-1}(\derR\sHom(\Sigma^\infty(\mathbf{G}_{m +}),C))
\To \Sigma^{-1}\Sigma U(C)=U(C)$$
(o\`u $C\To\derR\sHom(T,C)$ est obtenu par transposition de $w$, et o\`u
l'inclusion d\'esigne le morphisme induit par la d\'ecomposition
de $\Sigma^\infty(\mathbf{G}_{m +})$ en facteurs directs $\Sigma^\infty(\mathbf{G}_{m +})=
\Sigma^\infty(S_{+})\vee \Sigma^{-1}(T)$, obtenue via la section unit\'e de $\mathbf{G}_m$).
Nous pouvons alors d\'efinir $F_k$ par la formule
$$F_{k-1}=U(F_k)\, .$$
On d\'efinit enfin la construction de Bass-Thomason-Trobaugh $E^\TT$ par la
formule
$$E^\TT=\derL\underset{n\geq 0}{\varinjlim} F_{-n}\, .$$
On a alors, par construction, un morphisme canonique
$$E\To E^\TT\, .$$
La construction de $E^{\sharp}$ est plus directe.
Le foncteur $T\wedge(-)$ induit des morphismes canoniques
$$\derR\sHom(T^{\wedge n},E)\To\derR\sHom(T^{\wedge n+1},T\wedge^\derL E)\, ,$$
et le morphisme $w$ induit des morphismes \'evidents
$$w_*:\derR\sHom(T^{\wedge n+1},T\wedge^\derL E)\To\derR\sHom(T^{\wedge n+1}, E)\, .$$
On obtient donc une suite de morphismes
$$E\To\derR\sHom(T, E)\To\cdots\To\derR\sHom(T^{\wedge n}, E)
\To \derR\sHom(T^{\wedge n+1}, E)\To\cdots$$
On pose enfin
$$E^\sharp=\derL\underset{n\geq 0}{\varinjlim}\, \derR\sHom(T^{\wedge n}, E)\, .$$
On dispose aussi, par construction, d'un morphisme canonique
$$E\To E^\sharp\, .$$
\end{paragr}

\begin{paragr}
Si on consid\`ere le couple  $(K,b)$ correspondant \`a la
$K$-th\'eorie (voir \ref{Bott}), la construction de Bass-Thomason-Trobaugh
nous donne, de mani\`ere tautologique (en regard de la construction
donn\'ee dans \cite[preuve du lemme 6.3]{TT}), le r\'esultat de repr\'esentabilit\'e suivant.
\end{paragr}

\begin{prop}\label{repTTpref}
Le spectre $K^\TT$ repr\'esente la $K$-th\'eorie de Bass-Tho\-ma\-son-Tro\-baugh.
Autrement dit, pour tout $S$-sch\'ema lisse $X$, et pour tout entier $n$,
on a un isomorphisme de groupes ab\'eliens
$$\Hom_{\ho(\Sp_{S^1}(S))}(\Sigma^n\Sigma^\infty(X_+),K^\TT)\simeq K^\TT_n(X)\, ,$$
o\`u $K^\TT_n(X)$ d\'esigne le $n$-\`eme groupe de $K$-th\'eorie de Bass au sens
de Thomason et Trobaugh~\cite[D\'efinition 6.4]{TT}.
\end{prop}

\begin{paragr}
Le spectre de \emph{$K$-th\'eorie invariante par homotopie} (au sens de Weibel~\cite{Wei1,TT})
est, par d\'efinition:
$$\KH=R_{\AA^1}(K^\TT)\, .$$
On a donc, pour tout $S$-sch\'ema lisse $X$, et tout entier $n$, un isomorphisme
de groupes
$$\Hom_{\ho(\Sp_{S^1}(S))}(\Sigma^n \Sigma^\infty(X_+),\KH)\simeq\KH_n(X)\, .$$
Afin de comprendre la $K$-th\'eorie invariante par homotopie
au sein de la th\'eorie de l'homotopie des sch\'emas, nous allons
comparer le spectre $\KH$ et le spectre $\KV^\sharp$.
\end{paragr}

\begin{prop}\label{A1invTT2}
Soient $E$ et $F$ deux objet de $\Sp_{S^1}(S)$, munis
de morphismes $w:T\wedge E\To E$ et $w':T\wedge F\To F$.
On suppose donn\'e un morphisme $T$-\'equivariant $\varphi:E\To F$.
Si le morphisme $\varphi:E\To F$ est une $\AA^1$-\'equivalence,
alors il en est de m\^eme des morphismes
induits $\varphi^\TT:E^\TT \To F^\TT$ et $\varphi^\sharp: E^\sharp\To F^\sharp$.
\end{prop}

\begin{proof}
Cela r\'esulte imm\'ediatement du lemme \ref{A1eqHomint}.
\end{proof}

\begin{prop}\label{A1invTT}
Sous les hypoth\`eses de \ref{consTT},
si $E$ est invariant par homotopie, alors il en est de m\^eme
de $E^\TT$ et de $E^\sharp$, et on a alors un isomorphisme
canonique
$$E^\TT\simeq E^\sharp\, .$$
\end{prop}

\begin{proof}
Les spectres invariants par homotopie forment une
sous-cat\'egorie localisante de $\ho(\Sp_{S^1}(S))$,
et, si un spectre $F$ est invariant par homotopie, il en est
de m\^eme de $\derR\sHom(C,F)$ pour tout objet $C$
de $\ho(\Sp_{S^1}(S))$. On en d\'eduit aussit\^ot, par examen
des constructions de $E^\TT$ et de $E^\sharp$, que
ces derniers sont invariants par homotopie d\`es que c'est le cas
pour $E$.

Si $C$ est invariant par homotopie, on voit imm\'ediatement que
l'objet $V(C)$ construit au num\'ero \ref{consTT} est canoniquement
isomorphe \`a $C$, de sorte que le triangle distingu\'e
$$U(C)\To V(C)\To\derR\sHom(\Sigma^\infty(\mathbf{G}_{m +}),C)\To \Sigma U(C)$$
s'identifie au triangle distingu\'e
$$\derR\sHom(T,C)\To C \To \derR\sHom(\Sigma^\infty(\mathbf{G}_{m +}),C) \To \Sigma \derR\sHom(T,C)$$
(correspondant \`a la d\'ecomposition
$\Sigma^\infty(\mathbf{G}_{m +})=\Sigma^\infty(S_{+})\vee \Sigma^{-1} T$).
Si, en outre, on a un morphisme $T\wedge^\derL C\To C$,
sous ces identifications, le morphisme $C\To U(C)$ n'est autre que le morphisme
$C\To\derR\sHom(T,C)$ induit par adjonction.
En appliquant ce qui pr\'ec\`ede aux objets $C=\derR\sHom(T^{\wedge n},E)$,
on en d\'eduit que les spectres $E^\TT$ et $E^\sharp$ sont canoniquement isomorphes.
\end{proof}

\begin{cor}\label{A1invTT3}
Sous les hypoth\`eses de \ref{consTT}, on a des isomorphismes
canoniques dans $\ho(\Sp_{S^1}(S))$:
$$R_{\AA^1}(E^\TT)\simeq R_{\AA^1}(E)^\TT\simeq
R_{\AA^1}(E)^\sharp\simeq R_{\AA^1}(E^\sharp) \, .$$
\end{cor}

\begin{proof}
En vertu de la propositions \ref{A1invTT2} et de la premi\`ere assertion de la
proposition \ref{A1invTT}, le morphisme
$E^\TT\To R_{\AA^1}(E)^\TT$ (resp. $E^\sharp\To R_{\AA^1}(E)^\sharp$)
est une $\AA^1$-\'equivalence dont le but est invariant par homotopie,
et donc son image par le foncteur $R_{\AA^1}$ est un isomorphisme de m\^eme but
(\`a isomorphisme canonique pr\`es).
Ce corollaire r\'esulte donc de l'identification de $R_{\AA^1}(E)^\TT$
et de $R_{\AA^1}(E)^\sharp$, donn\'ee par la seconde assertion de la proposition \ref{A1invTT}.
\end{proof}

\begin{cor}\label{A1invTT4}
Il existe des isomorphismes canoniques $\KH\simeq \KV^\TT\simeq\KV^\sharp$.
\end{cor}

\begin{paragr}\label{def:descNis}
On rappelle qu'un pr\'efaisceau de $S^1$-spectres $E$ sur la cat\'egorie des $S$-sch\'emas lisses
a la \emph{propri\'et\'e de descente relativement \`a la topologie de Nisnevich}
si $E(\varnothing)\simeq 0$, et si, pour tout carr\'e cart\'esien de $S$-sch\'emas lisses
$$\xymatrix{
U\times_X V\ar[r]\ar[d]&V\ar[d]^f\\
U\ar[r]_j&X
}$$
avec $j$ une immersion ouverte, et $f$ un morphisme \'etale, induisant
un isomorphisme $f^{-1}(X-U)_{\mathrm{r\acute{e}d}}\simeq (X-U)_{\mathrm{r\acute{e}d}}$,
le carr\'e commutatif
$$\xymatrix{
E(X)\ar[r]\ar[d]& E(V)\ar[d]\\
E(U)\ar[r]& E(U\times_X V)
}$$
est homotopiquement cart\'esien (on rappelle que cette condition
est \'equivalente \`a la propri\'et\'e de descente cohomologique
formul\'ee en terme d'hyper-re\-couv\-re\-ments de Nisnevich; voir \cite{voe2,voe3}).

On v\'erifie facilement que si $E$ v\'erifie la propri\'et\'e de descente relativement \`a la
topologie de Nisnevich, alors il en est de m\^eme de $\derR\sHom(C,E)$
pour tout pr\'efaisceau de $S^1$-spectres $C$ (il suffit de le v\'erifier
dans le cas o\`u $C=\Sigma^\infty(X_+)$ pour $X$ lisse sur $S$).
En outre, les pr\'efaisceaux de $S^1$-spectres
v\'erifiant la propri\'et\'e de descente relativement \`a la
topologie de Nisnevich forment une sous-cat\'egorie localisante de
$\ho(\Sp_{S^1}(S))$. Cela implique que, si $E$ v\'erifie la propri\'et\'e de descente relativement \`a la
topologie de Nisnevich, il en est de m\^eme de $R_{\AA^1}(E)$,
ainsi que, lorsque cela a un sens, de $E^\TT$ et de $E^\sharp$.
\end{paragr}

\begin{cor}\label{NisdescKV}
Le spectre $\KV^\sharp$ v\'erifie la propri\'et\'e de descente relativement \`a la
topologie de Nisnevich.
\end{cor}

\begin{proof}
En vertu des th\'eor\`emes d'excision et de localisation \cite[7.1 et 7.4]{TT},
on sait d\'ej\`a que le spectre $K^\TT$ (et donc aussi, d'apr\`es
ce qui pr\'ec\`ede, $\KH$) v\'erifie la propri\'et\'e de descente relativement \`a la
topologie de Nisnevich
(on pourrait aussi invoquer directement \cite[Th\'eor\`eme 10.8]{TT}).
Le corollaire r\'esulte donc de l'identification de $\KV^\sharp$
avec $\KH$.
\end{proof}

\begin{paragr}
Soit $\Sp_T\Sp_{S^1}(S)$ la cat\'egorie des $T$-spectres dans la
cat\'egorie des pr\'efaisceaux de $S^1$-spectres $\Sp_{S^1}(S)$.
Les objets de $\Sp_T\Sp_{S^1}(S)$
sont des collections $E=(E_n,\sigma_n)_{n\geq 0}$, o\`u, pour $n\geq 0$,
$E_n$ est un objet de $\Sp_{S^1}(S)$, et $\sigma_n:T\wedge E_n\To E_{n+1}$
est un morphisme de $S^1$-spectres. On d\'efinit, \`a partir de la
structure de cat\'egorie de mod\`eles stable sur $\Sp_{S^1}(S)$, une
structure de cat\'egorie de mod\`eles $T$-stable sur $\Sp_T\Sp_{S^1}(S)$,
de sorte la cat\'egorie homotopique $\ho(\Sp_T\Sp_{S^1}(S))$
est canoniquement munie d'une structure de cat\'egorie triangul\'ee; voir
\cite{hov,ay2}. On note 
$$\Omega^\infty_T:\Sp_T\Sp_{S^1}(S)\To \Sp_{S^1}(S)$$
le foncteur d'\'evaluation en z\'ero $E\longmapsto E_0$. C'est un
foncteur de Quillen \`a droite, et donc, sont adjoint \`a gauche,
$$\Sigma^\infty_T: \Sp_{S^1}(S)\To \Sp_T\Sp_{S^1}(S)$$
est un foncteur de Quillen \`a gauche. On a donc une adjonction
d\'eriv\'ee:
$$\derL\Sigma^\infty_T:\ho(\Sp_{S^1}(S))\rightleftarrows\ho(\Sp_T\Sp_{S^1}(S)):\derR\Omega^\infty_T\, .$$
Par construction de $\ho(\Sp_T\Sp_{S^1}(S))$, le smash produit par $T$ est une \'equi\-va\-lence
de cat\'egories, ce qui donne un sens \`a l'expression $T^{\wedge n}\wedge^\derL E$
pour tout entier $n<0$.
\'Etant donn\'ee une propri\'et\'e $\mathcal{P}$ portant sur les objets de $\ho(\Sp_{S^1}(S))$,
on dira qu'un objet $E$ de $\ho(\Sp_T\Sp_{S^1}(S))$ a la propri\'et\'e
$\mathcal{P}$ si, pour tout entier $n$, le pr\'efaisceau en $S^1$-spectres
$\derR\Omega^\infty_T(T^{\wedge n}\wedge^\derL E)$ a la propri\'et\'e
$\mathcal{P}$ dans $\ho(\Sp_{S^1}(S))$.

On d\'esigne par $\SH(S)$ la sous-cat\'egorie pleine de $\ho(\Sp_T\Sp_{S^1}(S))$
form\'ee des objets v\'erifiant la propri\'et\'e d'invariance par homotopie
ainsi que la propri\'et\'e de descente relativement \`a la topologie de Nisnevich.
Le foncteur d'inclusion
$\SH(S)\To\ho(\Sp_T\Sp_{S^1}(S))$ admet un adjoint \`a gauche que nous noterons
$$\gamma:\ho(\Sp_T\Sp_{S^1}(S))\To\SH(S)\, .$$
On v\'erifie ais\'ement (par comparaison des propri\'et\'es universelles) que
la ca\-t\'e\-go\-rie $\SH(S)$ est canoniquement \'equivalente \`a la cat\'egorie homotopique
stable des sch\'emas construite en termes de $T$-spectres ou encore de $\PP^1$-spectres
dans la litt\'erature~\cite{Jar,riou1,ay2}.
\end{paragr}

\begin{paragr}\label{constrKGL}
Soit $E$ un pr\'efaisceau de $S^1$-spectres sur la cat\'egorie des $S$-sch\'emas lisses,
muni d'un morphisme $w:T\wedge E\To E$. On lui associe un $T$-spectre
$$\underline{E}=(E_n,\sigma_n)_{n\geq 0}$$
en posant $E_n=E$ et $\sigma_n=w$ pour tout $n\geq 0$.
Le morphisme $w$ induit un morphisme
$$\underline{w}:T\wedge^\derL\underline{E}\To\underline{E}$$
dans $\ho(\Sp_T\Sp_{S^1}(S))$, lequel
s'av\`ere \^etre un isomorphisme.
En outre, on a alors un isomorphisme canonique
$$E^\sharp\simeq\derR\Omega^\infty_T(\underline{E})$$
dans la cat\'egorie $\ho(\Sp_{S^1}(S))$ (cela r\'esulte par exemple
de \cite[Propositions 4.6 et 4.7]{hov}, ou bien encore de \cite[Th\'eor\`eme 4.3.61]{ay2}).
Il en d\'ecoule que, \'etant donn\'ee une propri\'et\'e raisonnable $\mathcal{P}$
des objets de $\ho(\Sp_{S^1}(S))$ (par exemple, la propri\'et\'e de descente pour
une topologie $t$, ou bien la propri\'et\'e d'invariance par homotopie), pour que $\underline{E}$
v\'erifie la propri\'et\'e $\mathcal{P}$ en tant qu'objet de $\ho(\Sp_T\Sp_{S^1}(S))$,
il faut et il suffit que $E^\sharp$ v\'erifie la propri\'et\'e $\mathcal{P}$
en tant qu'objet de $\ho(\Sp_{S^1}(S))$.

En consid\'erant le spectre de $K$-th\'eorie (resp. le spectre de $K$-th\'eorie
invariante par homotopie na\"\i ve) muni du cup produit par la classe $b$
(cf. \ref{Bott}), on obtient donc un objet $\underline{K}$ (resp. $\underline{\KV}$)
dans $\ho(\Sp_T\Sp_{S^1}(S))$. On a ainsi des isomorphismes:
$$K^\sharp\simeq\derR\Omega^\infty_T(\underline{K})
\quad\text{et}\quad \KH\simeq\KV^\sharp\simeq\derR\Omega^\infty_T(\underline{\KV})\, .$$
On remarque que, en vertu des corollaires \ref{A1invTT3} et \ref{NisdescKV},
le $T$-spectre $\underline{\KV}$ v\'erifie les propri\'et\'es
de descente relativement \`a la topologie de Nisnevich
et d'invariance par homotopie, et qu'il repr\'esente la $K$-th\'eorie invariante par
homotopie dans $\SH(S)$.

On d\'efinit par ailleurs le $T$-spectre de $K$-th\'eorie $\KGL$ par la
formule
$$\KGL=\gamma(\underline{K})\, .$$
\end{paragr}

\begin{rem}\label{P1spectrekth}
Lorsqu'on applique le foncteur d'espace de lacets infini au pr\'efaisceau
de $K$-th\'eorie, on obtient un pr\'efaisceau de complexes de Kan point\'e
$\derR\Omega^\infty(K)$ sur la cat\'egorie des $S$-sch\'emas lisses, lequel
est le foncteur de $K$-th\'eorie \`a valeurs dans les espaces point\'es
(par opposition aux $S^1$-spectres). Le pr\'efaisceau $\derR\Omega^\infty(K)$
associe \`a un un $S$-sch\'ema $X$ le complexe de Kan point\'e
$\derR\Omega(wS\, \mathit{Perf}(X))$,
correspondant \`a l'espace des lacets de la construction de Waldhausen appliqu\'ee
\`a la cat\'egorie des complexes parfaits sur $X$.
Il r\'esulte du th\'eor\`eme de Gillet-Waldhausen \cite[Th\'eor\`eme 1.11.7]{TT}
que, pour tout $S$-sch\'ema lisse de type fini $X$, le morphisme
canonique de $iS\, \mathit{Vect}(X)$ (la construction de Waldhausen
appliqu\'ee \`a la cat\'egorie exacte des fibr\'es vectoriels sur $X$)
vers $wS\, \mathit{Perf}(X)$ est une \'equivalence faible simpliciale localement
pour la topologie de Zariski (et donc de Nisnevich):
pour avoir une \'equivalence faible globalement,
il suffit que $X$ admette une famille ample de fibr\'es en droites, puisqu'alors
les complexes parfaits sur $X$ s'identifient aux complexes born\'es de $\mathcal{O}_X$-modules
localement libres de rang fini; cf. \cite[Corollaire 3.9]{TT}.
D'autre part, le morphisme canonique de $B(\amalg_{n\geq 0}\BGL_n)$
vers le pr\'efaisceau simplicial
$iS\, \mathit{Vect}$ (correspondant \`a l'inclusion de la cat\'egorie des
fibr\'es vectoriels triviaux dans $\mathit{Vect}$) est lui aussi une
\'equivalence faible simpliciale localement pour la topologie de Zariski. Enfin, en vertu
de \cite[Proposition 3.10, page 139]{MV}, on a une $\AA^1$-\'equivalence
de $\mathbf{Z}\times \BGL_\infty$ vers $\derR\Omega(B(\amalg_{n\geq 0}\BGL_n))$.
Autrement dit, on a un isomorphisme canonique
$$\mathbf{Z}\times \BGL_\infty\simeq \derR\Omega^\infty(K)$$
dans la cat\'egorie homotopique (instable) de Morel et Voevodsky.


Ce qui est d\'esign\'e habituellement comme le $\PP^1$-spectre de $K$-th\'eorie
en th\'eo\-rie de l'homotopie des sch\'emas \cite{riou2,riou3,PPR,SO}
admet la description suivante\footnote{Nous insistons sur le fait que ce $\PP^1$-spectre
de $K$-th\'eorie n'est pas d\'efini en tant qu'il repr\'esente
quoi que ce soit dans la cat\'egorie homotopique stable des sch\'emas: il faut plut\^ot
le voir comme un analogue purement formel du spectre de $K$-th\'eorie topologique en th\'eorie
de l'homotopie des sch\'emas. On verra plus loin que ce spectre repr\'esente en fait
la $K$-th\'eorie invariante par homotopie;
cf. proposition \ref{prop:P1spectrekth} et th\'eor\`eme \ref{thmvoe}.}.
La classe de Bott $\beta=[\mathcal{O}_{\PP^1}]-[\mathcal{O}_{\PP^1}(-1)]$
dans le groupe
$K_0(\PP^1)=\pi_0(\derR\Omega^\infty(K)(\PP^1))$ d\'efinit un morphisme
$\beta:\PP^1\To \derR\Omega^\infty(K)$ dans $\ho(\esp_\bullet(S))$, et induit donc un morphisme
$$\beta:\PP^1\To \derR\Omega^\infty(K)\simeq \mathbf{Z}\times \BGL_\infty$$
dans la cat\'egorie homotopique instable point\'ee $\mathcal{H}_\bullet(S)$.
Le $\PP^1$-spectre de $K$-th\'eorie usuel, que nous noterons ici $\mathcal{K}$, est
le $\PP^1$-spectre p\'eriodique
d\'etermin\'e par le cup produit par la classe $\beta$, c'est-\`a-dire le
$\PP^1$-spectre d\'etermin\'e par la collection de pr\'efaisceaux simpliciaux
$$(\mathbf{Z}\times \BGL_\infty,\mathbf{Z}\times \BGL_\infty,\ldots,\mathbf{Z}\times \BGL_\infty,\ldots)$$
avec
$\beta\cup:\PP^1\wedge (\mathbf{Z}\times \BGL_\infty)\To \mathbf{Z}\times \BGL_\infty$
pour morphismes structuraux (o\`u on a pris implicitement un remplacement fibrant de $\mathbf{Z}\times \BGL_\infty$ et un remplacement cofibrant de $\PP^1$
dans $\esp_\bullet(S)$).

Lorsqu'on travaille localement pour la topologie de Nisnevich (en fait, Zariski suffit)
et modulo $\AA^1$-\'equivalence, on a l'identification $S^1\wedge\Gm\simeq \PP^1$
(o\`u $\PP^1$ est consid\'er\'e comme un espace point\'e).
Cela permet de d\'ecrire la cat\'egorie $\SH(S)$ en termes de $\PP^1$-spectres;
cf. \cite[Th\'eor\`eme 4.3.40]{ay2}. C'est ce qui donne un sens \`a l'\'enonc\'e suivant.
\end{rem}

\begin{prop}\label{prop:P1spectrekth}
L'\'equivalence de cat\'egories entre la cat\'egorie homotopique stable
des $T$-spectres et la cat\'egorie homotopique stable des $\PP^1$-spectres
envoie $\KGL$  sur $\mathcal{K}$.
\end{prop}

\begin{proof}
La cat\'egorie homotopique stable des $\PP^1$-spectres
de pr\'e\-fais\-ceaux
simpliciaux est canoniquement \'equivalente \`a celle des $\PP^1$-spectres
de $S^1$-spectres. Comme $\mathbf{Z}\times \BGL_\infty$ et $\derR\Omega^\infty(K)$
sont $\AA^1$-\'equivalents, cette \'equivalence de cat\'egories identifie le $\PP^1$-spectre $\mathcal{K}$
introduit au num\'ero \ref{P1spectrekth} avec le $\PP^1$-spectre donn\'e par la collection
de $S^1$-spectres $(K,K,\ldots,K,\ldots)$ munie des morphismes structuraux
$\beta\cup:\PP^1\wedge K\To K$ correspondant au cup produit par la classe
$\beta=[\mathcal{O}_{\PP^1}]-[\mathcal{O}_{\PP^1}(-1)]$.
Il suffit donc de voir que l'identification $\PP^1\simeq S^1\wedge\Gm$
identifie (au signe pr\`es) la classe $\beta$ ci-dessus avec la classe $b$
introduite au num\'ero \ref{Bott}. En outre,
il suffit de traiter le cas o\`u $S=\mathrm{Spec}\, \mathbf{Z}$
(par fonctorialit\'e de la $K$-th\'eorie, puisque les deux
classes $b$ et $\beta$ sont bien d\'efinie sur $\mathbf{Z}$).
Or, dans ce cas, la classe $b$ correspond au choix d'un g\'en\'erateur
de la partie libre du groupe ab\'elien
$$
K_1(\mathbf{Z}[t,t^{-1}])
\simeq K_1(\mathbf{Z})\oplus K_0(\mathbf{Z})
\simeq \mathbf{Z}/2\mathbf{Z}\oplus \mathbf{Z} \, ,
$$
et la classe $\beta$ induit un isomorphisme canonique
$$K_0(\PP^1_\mathbf{Z})\simeq K_0(\mathbf{Z})\oplus\beta K_0(\mathbf{Z})\simeq\mathbf{Z}\oplus\mathbf{Z}\, .$$
En outre, il est bien connu que le Bockstein $\partial$ dans la suite exacte fondamentale
$$0\To K_1(\mathbf{Z})\To K_1(\mathbf{Z}[t])\oplus K_1(\mathbf{Z}[t^{-1}])\To K_1(\mathbf{Z}[t,t^{-1}])
\overset{\partial}{\To} K_0(\mathbf{Z})\To 0$$
admet une section induite par la classe $b$.
Comme la suite exacte de Mayer-Vietoris
$$K_1(\mathbf{Z}[t])\oplus K_1(\mathbf{Z}[t^{-1}])\To K_1(\mathbf{Z}[t,t^{-1}])\To K_0(\PP^1_\mathbf{Z})$$
identifie l'image de $K_1(\mathbf{Z}[t,t^{-1}])$ avec $\beta K_0(\mathbf{Z})$,
on en d\'eduit aussit\^ot la proposition.
\end{proof}

\begin{prop}\label{compKGLKV}
Les $T$-spectres $\KGL$ et $\underline{\KV}$
sont canoniquement isomorphes dans $\SH(S)$.
\end{prop}

\begin{proof}
Comme $\underline{\KV}$ v\'erifie les propri\'et\'es
de descente relativement \`a la topologie de Nisnevich
et d'invariance par homotopie, on a un isomorphisme canonique
$\gamma(\underline{\KV})\simeq\underline{\KV}$.
Le morphisme $\underline{K}\To\underline{\KV}$ \'etant une $\AA^1$-\'equi\-va\-lence
terme \`a terme, il induit, d'apr\`es \cite[Lemme 4.3.59]{ay2},
un isomorphisme apr\`es application
du foncteur de localisation
$\KGL=\gamma(\underline{K})\simeq\gamma(\underline{\KV})\simeq\underline{\KV}$,
ce qui implique l'assertion.
\end{proof}


\begin{thm}[Voevodsky]\label{thmvoe}
Le $T$-spectre $\KGL$
repr\'esente la $K$-th\'eorie invariante par homotopie dans $\SH(S)$: pour tout $S$-sch\'ema
lisse $X$, et tout entier $n$, on a un isomorphisme de groupes
$$\Hom_{\SH(S)}(\Sigma^n \Sigma^\infty_T(X_+),\KGL)\simeq \KH_n(X)\, .$$
\end{thm}

\begin{proof}
Cela d\'ecoule aussit\^ot de la proposition pr\'ec\'edente
et du num\'ero \ref{constrKGL}.
\end{proof}

\begin{rem}
Ce th\'eor\`eme de repr\'esentabilit\'e
permet de d\'ecrire la $K$-th\'eorie invariante par homotopie
comme la th\'eorie cohomologique orient\'ee universelle
avec loi de groupe formelle multiplicative; voir~\cite{SO,PPR}.
Il permet aussi de d\'ecrire la $K$-th\'eorie invariante par homotopie comme
la th\'eorie cohomologique repr\'esent\'ee par le $T$-spectre de Snaith
$\Sigma^\infty_T(\PP^\infty_+)[\beta^{-1}]$ dans $\SH$; voir \cite{SO,GS}.
\end{rem}

\begin{rem}\label{rem:repBGLHS}
Bien que le th\'eor\`eme \ref{thmvoe} montre que le $T$-spectre de $K$-th\'eorie $\KGL$
repr\'esente la $K$-th\'eorie invariante par homotopie dans $\SH(S)$,
lorsque $S$ n'est pas r\'egulier,
nous ne savons rien de ce que l'objet $\mathbf{Z}\times\BGL_\infty$
repr\'esente dans la cat\'egorie homotopique instable $\mathcal{H}(S)$
(on s'attend cependant \`a ce que cela ait un rapport avec la $K$-th\'eorie
de Karoubi-Villamayor).
\end{rem}

\begin{cor}\label{repKTTcarp}
Si $q\geq 1$ est nilpotent dans $\mathcal{O}_S$, alors, pour tout $S$-sch\'ema lisse $X$,
et pour tout entier $n$, on a un isomorphisme de groupes
$$\Hom_{\SH(S)}(\Sigma^n \Sigma^\infty_T(X_+),\KGL)\otimes\mathbf{Z}[1/q]
\simeq K^\TT_n(X)\otimes\mathbf{Z}[1/q]\, .$$
\end{cor}

\begin{proof}
Il s'agit d'une cons\'equence imm\'ediate du th\'eor\`eme
pr\'e\-c\'e\-dent et de \cite[Th\'eor\`eme 9.6]{TT}.
\end{proof}

\section{Descente par \'eclatements abstraits}

\begin{paragr}
En vertu de \cite[Section 4.5]{ay2},
les cat\'egories homotopiques stables $\SH(S)$ forment un
$2$-foncteur homotopique stable $\SH$ au sens de \cite[D\'efinition 1.4.1]{ay1}
(et m\^eme un d\'erivateur alg\'ebrique homotopique stable au sens de \cite[D\'efinition 2.4.13]{ay1}).
\'Etant donn\'e un morphisme de sch\'emas $f:S'\To S$, on a donc un couple de foncteur
adjoints
$$\derL f^*:\SH(S)\rightleftarrows\SH(S'):\derR f_*$$
(avec $\derL f^*$ adjoint \`a gauche de $\derR f_*$).
Le foncteur $\derL f^*$ est essentiellement d\'etermin\'e par le fait
qu'il commute aux colimites homotopiques et qu'il correspond au foncteur
de changement de base par $f$: pour tout $S$-sch\'ema lisse $X$,
en posant $X'=S'\times_S X$, on a:
$$\derL f^*\Sigma^\infty_T(X_+)=\Sigma^\infty_T(X'_+)\, .$$
Lorsque $f$ est en outre lisse, le foncteur $\derL f^*$
a aussi un adjoint \`a gauche
$$\derL f_\sharp:\SH(S')\To\SH(S)$$
essentiellement d\'etermin\'e par le fait que, pour tout
$S'$-sch\'ema lisse $X$, on a
$$\derL f_\sharp\Sigma^\infty_T(X_+)=\Sigma^\infty_T(X_+)\, .$$
Nous utiliserons de mani\`ere essentielle les faits suivants.
\end{paragr}

\begin{thm}[Localisation]\label{loc}
Soit $i:Z\To S$ une immersion ferm\'ee, d'immersion ouverte compl\'ementaire
$j:U\To S$. Pour tout objet $E$ de $\SH(S)$, le carr\'e commutatif
$$\xymatrix{
\derL j_\sharp \derL j^*(E)\ar[r]\ar[d]& E\ar[d]\\
0\ar[r]& \derR i_*\derL i^*(E)
}$$
est homotopiquement cocart\'esien. Autrement dit, on a alors un triangle distingu\'e
canonique
$$\derL j_\sharp \derL j^*(E)\To E\To \derR i_*\derL i^*(E)\To
\Sigma \, \derL j_\sharp \derL j^*(E)\, .$$
En outre, les foncteurs
$$\derL j_\sharp:\SH(U)\To\SH(S)\quad\text{et}\quad
\derR i_*:\SH(Z)\To\SH(S)$$
sont pleinement fid\`eles.

\noindent En particulier, le foncteur
$$(\derL j^*,\derL i^*):\SH(S)\To\SH(U)\times\SH(Z)$$
est conservatif.
\end{thm}

\begin{proof}
Voir \cite[Section 4.5.3]{ay2}.
\end{proof}

\begin{cor}\label{red}
Pour tout sch\'ema $X$, l'immersion $i:X_{\mathrm{red}}\To X$
induit une \'equivalence de cat\'egories
$$\derL i^*:\SH(X)\To\SH(X_{\mathrm{red}})\, .$$
\end{cor}

\begin{proof}
Cela r\'esulte imm\'ediatement du th\'eor\`eme pr\'ec\'edent,
puis\-que l'ouvert compl\'ementaire de l'immersion ferm\'ee $i$ est
vide (sachant que $\SH(\varnothing)\simeq 0$).
\end{proof}

\begin{thm}[Changement de base lisse]\label{lisse}
Pour tout carr\'e cart\'esien dans la cat\'egorie des
sch\'emas,
$$\xymatrix{
X'\ar[r]^u \ar[d]_{q}&X\ar[d]^p\\
Y'\ar[r]_v&Y
}$$
si le morphisme $v$ est lisse, alors, pour tout objet $E$ de $\SH(Y')$,
le morphisme canonique
$$\derL u_\sharp \, \derL q^*(E)\To\derL p^* \, \derL v_\sharp(E)$$
est un isomorphisme dans $\SH(X)$.

\noindent Par transposition, pour tout objet $E$ de $\SH(X)$,
le morphisme canonique
$$\derL v^*\, \derR p_*(E)\To\derR q_* \, \derL u^*(E)$$
est un isomorphisme dans $\SH(Y')$.
\end{thm}

\begin{proof}
Voir \cite[Proposition 4.5.48]{ay2}.
\end{proof}

\begin{thm}[Changement de base propre]\label{propre}
Pour tout carr\'e cart\'esien de sch\'emas,
$$\xymatrix{
X'\ar[r]^u \ar[d]_{q}&X\ar[d]^p\\
Y'\ar[r]_v&Y
}$$
si le morphisme $p$ est propre, alors, pour tout objet $E$ de $\SH(X)$,
le morphisme canonique
$$\derL v^*\, \derR p_*(E)\To\derR q_* \, \derL u^*(E)$$
est un isomorphisme dans $\SH(Y')$.
\end{thm}

\begin{proof}
Voir \cite[Corollaire 1.7.18]{ay1} pour le cas o\`u $p$ est
projectif. Le cas g\'en\'eral en d\'ecoule gr\^ace au lemme
de Chow; voir \cite[Proposition 2.3.11]{CD3}.
\end{proof}

\begin{paragr}
On rappelle qu'un morphisme de sch\'emas $p:X'\To X$ est
un \emph{\'eclatement abstrait} de centre $Z$ si $p$
est propre, et si $Z$ est un sous-sch\'ema ferm\'e tel que le
morphisme induit
$$p^{-1}(X-Z)_{\mathrm{r\acute{e}d}}\To (X-Z)_{\mathrm{r\acute{e}d}}$$
soit un isomorphisme.
La topologie cdh est la topologie de Grothendieck sur la cat\'egorie
des sch\'emas engendr\'ee par les recouvrements de Nisnevich
et par les recouvrements de la forme $Z\amalg X'\To X$ pour
tout \'eclatement abstrait $X'\To X$ de centre $Z$.
Nous renvoyons le lecteur \`a \cite[Lemme 5.8 et Proposition 5.9]{SV}
(dont les \'enonc\'es et les preuves sont tout-\`a-fait valables en in\'egales caract\'eristiques)
pour une description civilis\'ee des recouvrement cdh (\`a raffinement pr\`es).

Un pr\'efaisceau de $S^1$-spectres $E$ sur la cat\'egorie
des sch\'emas v\'erifie la propri\'et\'e de descente
relativement \`a la topologie cdh si et seulement
s'il v\'erifie la propri\'et\'e de descente
relativement \`a la topologie de Nisnevich (\ref{def:descNis}) et si, pour tout
\'eclatement abstrait $p:X'\To X$ de centre $Z$, en posant $Z'=p^{-1}(Z)$,
le carr\'e commutatif \'evident
$$\xymatrix{
E(X)\ar[r]\ar[d] &E(X')\ar[d]\\
E(Z)\ar[r]&E(Z')}$$
est homotopiquement (co)cart\'esien; voir \cite{voe2,voe3}.
\end{paragr}

\begin{prop}\label{cdh}
Soit $p:X'\To X$ un \'eclatement abstrait de centre $Z$.
On consid\`ere le carr\'e cart\'esien de sch\'emas correspondant
ci-dessous.
$$\xymatrix{
Z'\ar[r]^k \ar[d]_{q}&X'\ar[d]^p\\
Z\ar[r]_i&X
}$$
On note enfin $r=pk=iq:Z'\To X$.
Alors, pour tout objet $E$ de $\SH(X)$, le carr\'e commutatif
$$\xymatrix{
E\ar[r]\ar[d] &\derR p_*\, \derL p^* E\ar[d]\\
\derR i_*\, \derL i^* E \ar[r]&\derR r_*\, \derL r^* E
}$$
est homotopiquement cocart\'esien.
\end{prop}

\begin{proof}
Soit $j:U=X-Z\To X$ l'immersion ouverte comp\-l\'e\-men\-taire de $i$.
Par le th\'eor\`eme de localisation et par le
th\'eor\`eme de changement de base lisse, l'image du carr\'e
consid\'er\'e par le foncteur $\derL j^*$ est le carr\'e
$$\xymatrix{
&\derL j^* E\ar[r]^{=} \ar[d] & \derL j^* E\ar[d]&\\
&0 \ar[r]& 0&,
}$$
et, de m\^eme, en vertu des th\'eor\`emes de localisation et de changement de base propre,
son image par le foncteur $\derL i^*$ est le carr\'e
$$\xymatrix{
&\derL i^* E\ar[r]\ar[d]_{=} & \derR q_* \, \derL q^* \, \derL i^* E\ar[d]^{=}&\\
&\derL i^* E \ar[r]& \derR q_* \, \derL q^* \, \derL i^* E&.
}$$
Les deux carr\'es ci-dessus \'etant trivialement homotopiquement cocart\'esiens,
et les foncteurs $\derL i^*$ et $\derL j^*$
formant une famille conservative de foncteurs exacts (\ref{loc}),
cela prouve la proposition.
\end{proof}

\begin{prop}\label{stabchangementbaseKGL}
Pour tout morphisme de sch\'emas $f:S'\To S$, le morphisme canonique
$$\derL f^*(\KGL)\To\KGL$$
est un isomorphisme dans $\SH(S')$.
\end{prop}

\begin{proof}
On a un isomorphisme canonique
$$\derL f^*(\mathbf{Z}\times\BGL_\infty)\simeq
\mathbf{Z}\times\BGL_\infty$$
dans la cat\'egorie homotopique instable $\mathcal{H}(S')$
(car $\mathbf{Z}\times\BGL_\infty$ est une colimite homotopique
de sch\'emas lisses de la forme $\mathit{GL}_{n_1}\times\cdots\times \mathit{GL}_{n_r}$).
La proposition r\'esulte donc aussit\^ot de la description de $\KGL$ comme le $\PP^1$-spectre
p\'eriodique associ\'e au morphisme de Bott
$\PP^1\To \mathbf{Z}\times\BGL_\infty$; voir la remarque \ref{P1spectrekth}
et la proposition \ref{prop:P1spectrekth}.
\end{proof}

\begin{thm}\label{cdhKH}
La $K$-th\'eorie invariante par homotopie v\'erifie
la propri\'et\'e de descente relativement \`a la topologie cdh.
\end{thm}

\begin{proof}
On sait que $\KH$ v\'erifie la propri\'et\'e de descente
relativement \`a la topologie de Nisnevich.
Pour v\'erifier la descente cdh, il suffit donc de montrer la
propri\'et\'e de Mayer-Vietoris relativement aux \'eclatements
abstraits; or, en vertu de la proposition pr\'ec\'edente et du th\'eor\`eme \ref{thmvoe},
cela r\'esulte de l'\'evaluation en $X$ du carr\'e homotopiquement (co)cart\'esien
de la proposition \ref{cdh} appliqu\'ee \`a $E=\KGL$.
\end{proof}

\begin{cor}
Pour tout entier $q>0$, la $K$-th\'eorie de Bass-Tho\-ma\-son Trobaugh
\`a coefficients dans $\mathbf{Z}[1/q]$ (resp. dans $\mathbf{Z}/q\mathbf{Z}$)
v\'erifie la propri\'et\'e de descente relativement \`a la topologie cdh
pour les $\mathbf{Z}/q\mathbf{Z}$-sch\'emas (resp. les
$\mathbf{Z}[1/q]$-sch\'emas).
\end{cor}

\begin{proof}
Cela r\'esulte du th\'eor\`eme \ref{cdhKH} et de \cite[Th\'eor\`eme 9.6]{TT}.
\end{proof}

En guise d'application, on en d\'eduit la forme faible suivante de la
conjecture de Weibel en caract\'eristique positive, sous l'hypoth\`ese
de l'existence locale de r\'esolutions des singularit\'es (on rappelle que cela
n'est connu qu'en caract\'eristique nulle).

\begin{thm}\label{weibel1}
Soit $k$ un corps. On suppose que $k$ admet une r\'esolution locale des
singularit\'es dans le sens o\`u, pour tout $k$-sch\'ema de type fini $X$,
il existe un recouvrement cdh $X'\To X$ avec $X'$ r\'egulier.
Alors, pour tout $k$-sch\'ema $S$ de dimension de Krull $\leq d$,
et pour tout $S$-sch\'ema lisse $X$, on a
$$\KH_i(X)=0\quad \text{pour tout $i<-d$.}$$
Si, sous les m\^emes hypoth\`eses, le corps $k$ est de caract\'eristique $p>0$, on a donc:
$$K^B_i(X)\otimes
\mathbf{Z}[1/p]=0\quad \text{pour tout $i<-d$.}$$
\end{thm}

\begin{proof}
Notons $E$ le pr\'efaisceau de $S^1$-spectres sur la cat\'egorie des
$S$-sch\'emas de type fini d\'efini par
$$U\longmapsto \KH(U\times_S X)\, .$$
On d\'eduit du th\'eor\`eme \ref{cdhKH} que,
pour tout $k$-sch\'ema de type fini $U$, on a, pour tout
entier $n$, un isomorphisme canonique
$$H^n_{\textrm{cdh}}(U,E)=\KH_{-n}(U\times_S X) \, ,$$
o\`u $H^*_{\textrm{cdh}}(U,E)$ d\'esigne l'hyper-cohomologie cdh
de $U$ \`a coefficients dans le faisceau cdh de $S^1$-spectres associ\'e \`a $E$.
On dispose donc de la suite spectrale ci-dessous
$$E^{p,q}_2=H^p_{\textrm{cdh}}(S,H^q(E)) \Rightarrow
\KH_{-p-q}(X) \, ,$$
o\`u $H^q(E)$ d\'esigne le pr\'efaisceau de groupes ab\'eliens
$U\longmapsto \KH_{-q}(U\times_S X)$, et o\`u $H^*_{\textrm{cdh}}(S,F)$
d\'esigne la cohomologie  de $S$ \`a coefficients dans le
faisceau cdh associ\'e \`a $F$.
La dimension cohomologique cdh \'etant major\'ee par la dimension
de Krull (voir l'appendice de \cite{SV}), on a $E^{p,q}_2=0$ d\`es que $p>d$, et
la suite spectrale ci-dessus converge fortement.
Or, pour $q>0$, le faisceau cdh associ\'e \`a $H^q(E)$
est isomorphe \`a z\'ero: comme on a suppos\'e que $k$ admet une
r\'esolution des singularit\'es locale, il suffit de v\'erifier que $\KH_{-q}(Y)=0$ pour
$Y$ r\'egulier, ce qui est bien connu. Cela implique que
$E^{p,q}_2=0$ d\`es que $p+q>d$, et ach\`eve donc la d\'emonstration
de la premi\`ere assertion.

Dans le cas tr\`es hypoth\'etique o\`u $k$ est un corps de caract\'eristique $p>0$
qui admet une r\'esolution des singularit\'e locale au sens ci-dessus, on obtient la
seconde assertion \`a partir de la premi\`ere gr\^ace \`a \cite[Th\'eor\`eme 9.6]{TT}.
\end{proof}

\bibliography{KHdescente}
\bibliographystyle{amsplain}
\end{document}